\documentclass{amsart}
\usepackage{amsmath, amsfonts, amssymb, enumerate, hyperref}

\usepackage{xypic, comment}

\newcommand{\mM}{\mathcal{M}}

\newcommand{\fp}{\mathfrak{p}}

\newcommand{\bfA}{\mathbf{A}}

\newcommand{\bfC}{\mathbf{C}}

\newcommand{\bfF}{\mathbf{F}}

\newcommand{\bfQ}{\mathbf{Q}}

\newcommand{\bfZ}{\mathbf{Z}}

\newcommand{\Oo}{\mathcal{O}}

\newcommand{\com}[1]{\vspace{5 mm}\par \noindent
\marginpar{\textsc{Comment}} \framebox{\begin{minipage}[c]{0.95
\textwidth} \tt #1 \end{minipage}}\vspace{5 mm}\par}

\newcommand{\ov}{\overline}

\newcommand{\be}{\begin{equation}}
\newcommand{\ee}{\end{equation}}
\newcommand{\bes}{\begin{equation*}}
\newcommand{\ees}{\end{equation*}}

\newcommand{\bs}{\begin{split}}
\newcommand{\es}{\end{split}}
\newcommand{\bss}{\begin{split*}}
\newcommand{\ess}{\end{split*}}

\newcommand{\bmat}{\left[ \begin{matrix}}
\newcommand{\emat}{\end{matrix} \right]}
\newcommand{\bsmat}{\left[ \begin{smallmatrix}}
\newcommand{\esmat}{\end{smallmatrix} \right]}

\newcommand{\bml}{\begin{multline}}
\newcommand{\eml}{\end{multline}}
\newcommand{\bmls}{\begin{multline*}}
\newcommand{\emls}{\end{multline*}}

\DeclareMathOperator{\Ext}{Ext}

\DeclareMathOperator{\Frob}{Frob}

\DeclareMathOperator{\GL}{GL}

\DeclareMathOperator{\Hom}{Hom}

\newcommand{\tr}{\textup{tr}\hspace{2pt}}

\usepackage[all]{xy}
\usepackage{amsthm}
\usepackage{amssymb, amsfonts}
\SelectTips{cm}{10}\UseTips
\theoremstyle{plain}
\newtheorem{thm}{Theorem}
\newtheorem{prop}[thm]{Proposition}
\newtheorem{cor}[thm]{Corollary}
\newtheorem{lemma}[thm]{Lemma}
\newtheorem{conj}[thm]{Conjecture}

\theoremstyle{definition}
\newtheorem{definition}[thm]{Definition}

\newtheorem{example}[thm]{Example}
\newtheorem{rem}[thm]{Remark}

\numberwithin{thm}{section}
\numberwithin{equation}{section}

\DeclareFontFamily{U}{wncy}{}
    \DeclareFontShape{U}{wncy}{m}{n}{<->wncyr10}{}
    \DeclareSymbolFont{mcy}{U}{wncy}{m}{n}
    \DeclareMathSymbol{\Sha}{\mathord}{mcy}{"58}

\newtheorem*{remark*}{Remark}

\begin{document}
\title[Oddness of residually reducible Galois representations]{Oddness of residually reducible Galois representations}
\author{Tobias Berger$^1$} 
\address{$^1$School of Mathematics and Statistics, University of Sheffield, Hicks Building, Hounsfield Road, Sheffield S3 7RH, UK.}




\date{\today}

\begin{abstract}
We show that suitable
congruences between polarized automorphic forms over a CM field always produce elements in the Selmer group for exactly the $\pm$-Asai (aka tensor induction) representation that is critical in the sense of Deligne. For this we relate the oddness of the associated polarized Galois representations (in the sense of the Bella{\"{\i}}che-Chenevier sign being $+1$) to the parity condition for criticality. Under an assumption similar to  Vandiver's conjecture this also provides evidence for the Fontaine-Mazur conjecture for polarized Galois representations of any even dimension.

\end{abstract}

\maketitle

\section{Introduction} 
Ribet proved in \cite{Ribet76} that, if a prime $p>2$ divides the numerator of the $k$-th Bernoulli number for $k \geq 4$ an even integer, then $p$ divides the order of the part of the class group of $\bfQ(\mu_p)$ on which ${\rm Gal}(\bfQ(\mu_p)/\bfQ)$ acts by the $(1-k)$-th power of the mod $p$ cyclotomic character $\ov{\epsilon}$. The argument (in fact a slight variance of Ribet's proof outlined in \cite{Khare00}) is as follows: One considers the weight $k$ Eisenstein series for ${\rm SL}_2(\bfZ)$ and proves the existence of a weight $k$ cuspidal eigenform that is congruent to the Eisenstein series. Ribet then shows that one can find a lattice in the irreducible odd  $p$-adic Galois representation $\rho_f$ associated to $f$ such that its mod $p$ reduction gives a non-split extension $\begin{pmatrix} 1&*\\0& \ov{\epsilon}^{k-1}\end{pmatrix}$. Note that this proves one direction of the Bloch-Kato conjecture for the Riemann zeta function (proven in \cite{BlochKato90} Theorem 6.1 (i)) in so far as it relates the $p$-divisibility of the value $\frac{\zeta(k)}{\pi^k}$ for the even integer $k$ to that of the order of the Selmer group for the critical motive $\bfQ(1-k)$. For odd $k$ this motive is not critical and Vandiver's conjecture predicts, in fact, that the corresponding Selmer group is trivial. We note that $\rho_f$ is polarized in the sense that $$\rho_f^{\vee} \cong \rho_f \otimes {\rm det}(\rho_f)^{-1} \cong \rho_f \otimes \epsilon^{1-k},$$ and that the characters occurring in the reduction get swapped under this polarization.

We generalize Ribet's strategy using polarized (essentially conjugate dual) Galois representations of CM fields to prove part of the Bloch-Kato conjecture for general Asai (or tensor induction) $L$-functions, our main result in Theorem \ref{general}. For the Asai $L$-functions induced from CM fields a similar parity condition to that for the Riemann zeta function is required for criticality. We demonstrate in Corollary \ref{cor} that this parity condition is linked to the oddness of the polarized Galois representation (where the representation is odd if the sign defined in \cite{BellaicheCheneviersign} equals $+1$, see Definition \ref{signdefn}).

By the Fontaine-Mazur conjecture in \cite{FontaineMazur} all  representations $\sigma:G_{\bfQ} \to \GL_2(\ov{\bfQ}_p)$ that are ``geometric'' (i.e. unramified at almost all primes and potentially semi-stable at $p$) are modular and therefore odd (in the sense that $\det(\sigma(\tilde c))=-1$ for any complex conjugation $\tilde c$), or they must be a Tate twist of an even representation with finite image. Calegari proved the following remarkable result about the non-existence of even geometric representations:
\begin{thm}[\cite{Calegari2012}] \label{cal}
Let $p>7$ be a prime and $\sigma :G_{\bfQ} \to \GL_2(\ov{\bfQ}_p)$ a continuous representation such that $\sigma$ is unramified at all but finitely many primes, and at $p$ is potentially semi-stable with distinct Hodge-Tate weights. Further assume that the residual representation $\ov{\sigma}$ is absolutely irreducible and not of dihedral type, and that $\ov{\sigma}|_{D_p}$ is not a twist of a representation of the form $\begin{pmatrix} \ov{\epsilon} & *\\0& 1 \end{pmatrix}$. Then $\sigma$ is odd.
\end{thm}

In the residually reducible case one can use Ribet's argument to prove the following:

\begin{prop}\label{baby}
Let $p>2$ and $\sigma:G_{\bfQ} \to \GL_2(\ov{\bfQ}_p)$ an irreducible continuous representation such that $\ov{\sigma}^{\rm ss}=\mathbf{1} \oplus \ov{\epsilon}^m$ for $m \in \bfZ_{>0}$. Assume that $\sigma$ is unramified outside of $p$, and either ordinary at $p$ or crystalline at $p$ with Hodge-Tate weights in  $[0,p-3]$. Then Vandiver's conjecture ($p \nmid \#{\rm Cl}(\bfQ(\mu_p)^+)$) implies that $\sigma$ is odd.
\end{prop}

\begin{proof}
By \cite{Ribet76} one can find a lattice such that the mod $p$ reduction is a non-split extension of the form $\begin{pmatrix} \mathbf{1}&*\\0& \ov{\epsilon}^m \end{pmatrix}$ that is unramified away from $p$.  One can show that this extension is split at $p$. This follows from Fontaine-Laffaille theory if $\sigma$ is ``short crystalline'' as in the statement, or from ordinarity, which implies that $\sigma|_{D_p}$ has an unramified quotient. 

This implies (see e.g. \cite{Khare00} p. 275) that $p \mid  \#{\rm Cl}(\bfQ(\mu_p))(\ov{\epsilon}^{-m})$. By Vandiver's conjecture $m$ has to be odd, which implies that $\det(\sigma)$ is odd.
\end{proof}

\begin{rem}
The assumptions on the ramification of $\sigma$ can be weakened, e.g. by allowing ramification at primes $\ell$ satisfying $\ell^m \not \equiv 1 \mod{p}$. 
 At $p$ crystallinity alone is not enough
to be able to deduce that the extension is split at $p$, as the following example (pointed out to us by Kevin Buzzard) shows: Consider $p=3$ and $\sigma=\rho_{\Delta,3}$ the $3$-adic Galois representation associated to the weight 12 level 1 cuspform $\Delta$. Then $\ov{\rho}_{\Delta,3}^{\rm ss}=\mathbf{1} \oplus \ov{\epsilon}$, but the class number of $\bfQ(\mu_3)$ is one. 
\end{rem}
We will prove in Corollary \ref{cor8.3} a higher-dimensional generalization of Proposition \ref{baby} (for residually reducible polarized Galois representations of any dimension) as a consequence of the contrapositive of our main result.

Given a representation $\rho:G_{K} \to {\rm GL}_n(E)$ for $K/K^+$ a CM field and $E$ a finite extension of $\bfQ_p$ one can define two extensions of $\rho \otimes \rho^c$ to $G_{K^+}$, which we denote by ${\rm As}^{\pm}(\rho)$ (for details see section \ref{Asaidefinition}).  Our main result  Theorem \ref{general}  concerns the Bloch-Kato conjecture for these representations and generalizes our previous result in \cite{Berger15} for Asai representations of $2$-dimensional representations of an imaginary quadratic fields to this setting, including extensions of adjoint representations of polarized regular motives as studied in \cite{Harris13}.  We explore in this paper the subtle connection between polarizations, signs and criticality that this  exposes. 

To explain this we need to introduce some more notation:
Let $\Psi: G_{K^+} \to E^*$ be a Hecke character and $R:G_K \to {\rm GL}_{2n}(E)$ an irreducible $G_K$-representation such that $R^{\vee} \cong R^c \otimes \Psi|_{G_K}$ with $\ov{R}^{\rm ss}=\ov{\rho} \oplus \ov{\rho}^{c \vee} \otimes \Psi^{-1}|_{G_K}$ (i.e. $R$ is residually reducible with two summands that get swapped under the polarization). Using Shapiro's lemma we relate in Lemma \ref{Selmerres} the $K^+$-Selmer group for ${\rm As}^{\pm \Psi(c_v)}(\ov{\rho})\otimes \Psi$ to the part of its $K$-Selmer group on which the action by a complex conjugation $c_v$ for $v \mid \infty$ is given by $\pm \Psi(c_v)$. In Lemma \ref{explicit} we show that on $H^1({\rm As}^{-\Psi(c_v)}(\ov{\rho})\otimes \Psi)$ the action by $c_v$ agrees with an involution $\perp_{c_v}$ corresponding to the polarization condition satisfied by $R$. A result of \cite{BellaicheChenevierbook} implies that the involution $\perp_{c_v}$ acts on the $G_K$-extension constructed from $R$ by the sign associated to $(R, \Psi, c_v)$ (introduced in Definition \ref{signdefn}). Combining these, we get in  Theorem \ref{general} a construction of an element in the Bloch-Kato Selmer group of ${\rm As}^{\pm}(\rho)\otimes \Psi$, where the sign is given by $-\Psi(c_v) {\rm sign}(R, \Psi, c_v)$.
 In Proposition \ref{prop} we prove that this ${\rm As}^{\pm}(\rho)\otimes \Psi$ is critical in the sense of Deligne (for characters $\Psi$ where this corresponds to a twist in the left hand side of the critical strip of ${\rm As}(\rho)$) if and only if ${\rm sign}(R, \Psi, c_v)=1$.
That automorphic Galois representations $R$ are odd in this sense has been proven in \cite{BellaicheCheneviersign} for Galois representations associated to regular algebraic cuspidal polarized representations of ${\rm GL}_m$ over CM fields. We expect it to hold more generally, see the discussion in Remark \ref{r64}.

The (non-critical) $G_{K^+}$-represention ${\rm As}^{\Psi(c_v)}(\rho)\otimes \Psi$ does, in fact,  not depend on $\Psi(c_v)$. This allows us in Corollary \ref{cor8.3} to deduce total oddness of certain residually reducible polarized Galois representations if we assume that the order of the Selmer group $H^1_{\Sigma^+}(K^+, {\rm As}^{\Psi(\tilde c)}(\rho)\otimes \Psi \otimes (E/\Oo))$ is not divisible by $p$ for a fixed choice $\tilde c$ of complex conjugation.

\section*{Acknowledgement}The author would like to thank Ga\"etan Chenevier, Neil Dummigan, Kris Klosin, Jack Thorne and Jacques Tilouine for helpful conversations related to the topics of this article. The author's research was supported by the EPSRC First Grant EP/K01174X/1. The author would also  like to thank ESI, Vienna, for their hospitality during the ``Arithmetic Geometry
and Automorphic Representations" workshop in April/May 2015, where part of this research was carried out.

\section{Notation}
If $L$ is a number field then we fix algebraic closures $\ov{L}$ and $\ov{L}_v$ for every place $v$ of $L$ and compatible embeddings $\ov{L} \hookrightarrow \ov{L}_{v} \hookrightarrow \bfC$ and write $D_{L_v}$ and $I_{L_v}$ (or $I_v$ if $L$ is understood) for the corresponding decomposition and inertia subgroups of $G_{L}:={\rm Gal}(\ov{L}/L)$. 

We fix a prime  $p>2$ and denote by $\epsilon:G_{L} \to \bfZ_p^*$ the $p$-adic cyclotomic character. If $v$ is a finite place of $L$, not dividing $p$, then its value on the arithmetic Frobenius element $\epsilon(\Frob_v)=q_v$, where the latter denotes the size of the residue field of $L$ at $v$. For a $p$-adic representation $V$ of $G_L$, we write $V(n):=V \otimes \epsilon^n$ for $n\in \bfZ$. For the coefficients of our $p$-adic representations, we will take $E$ to be a (sufficiently large) finite extension of $\bfQ_p$ inside $\ov{\bfQ}_p$ with ring of integers $\Oo$ and residue field $\bfF$. We fix a choice
of a  uniformizer $\varpi$.

We fix an isomorphism $\ov{\bfQ}_p \cong \bfC$.
If $\chi: \bfA_L^*/L^* \to \bfC^*$ is a Hecke character of type $(A_0)$ then we write $\chi^{\rm gal}$ for the associated Galois character $G_L \to \ov{\bfQ}_p^*$ (for details see e.g. \cite{Thorne15}, but we use the local Artin reciprocity maps mapping uniformizers to arithmetic Frobenius elements).

Let  $K/K^+$ be a CM field, and write $\chi_{K/K^+}:G_{K^+} \to \{ \pm 1\}$ for the corresponding quadratic character. We fix a complex conjugation $\tilde c \in G_{K^+}$. For a $G_K$-representation $(\rho, V)$ we define the conjugate representation $(\rho^c, V^c)$ by $\rho^c(g)=\rho(\tilde c g \tilde c)$.

\section{Asai $L$-function and tensor induction} \label{Asaidefinition}
We describe the definition of the Asai $L$-function of  a cuspidal automorphic representation $\pi$ for ${\rm GL}_m(K)$ for a CM field $K/K^+$ (adapting the treatment for $m=2$ in \cite{Krishna12}, see also \cite{Ramak04} Section 6). The Langlands dual of the algebraic group $R_{K/K^+} {\rm GL}_m$ is given by ${}^L(R_{K/K^+} {\rm GL}_m)=({\rm GL}_m(\bfC) \times {\rm GL}_m(\bfC)) \rtimes G_{K^+}$. There are $m^2$-dimensional representations
$$r^{\pm}:{}^L(R_{K/K^+} {\rm GL}_m/K)(\bfC)\to {\rm GL}(\bfC^m \otimes \bfC^m)$$  given by $$r^{\pm}(g,g',\gamma)(x \otimes y)=g(x) \otimes g'(y) \text{ for } \gamma|_K=1$$ and $$r^{\pm}(1,1,\gamma)(x \otimes y)=\pm y \otimes x.\text{ for } \gamma|_K\neq1$$ For each place $v$ we denote the local $L$-group homomorphisms obtained from $r^{\pm}$ by restriction by $r^{\pm}_v$.

Let $\pi$  be a cuspidal automorphic representation for ${\rm GL}_m(\bfA_K)$. Then we may consider $\pi$ as a representation of $R_{K/K^+} {\rm GL}_m(\bfA)$ and factorize it as a restricted tensor product $$\pi=\otimes_v \pi_v,$$ where each $\pi_v$ is an irreducible admissible representation of ${\rm GL}_m(K \otimes_{K^+} K^+_v)$, with corresponding $L$-parameter $\phi_v: W'_{K^+_v} \to ({\rm GL}_m(\bfC) \times {\rm GL}_m(\bfC)) \rtimes G_{K^+_v}$. Let ${\rm As}^{\pm}(\pi_v)$ now be the irreducible admissible representation of ${\rm GL}_{m^2}(K^+_v)$ corresponding to the parameter $r^{\pm}_v \circ \phi_v$ under the local Langlands correspondence and put ${\rm As}^{\pm}(\pi):=\otimes_v {\rm As}^{\pm}(\pi_v)$. By Langlands functoriality these should be  isobaric automorphic representations of ${\rm GL}_{m^2}(\bfA)$, the $\pm$ Asai transfers  of $\pi$. For $m=2$ this has been proven by Krishnamurty \cite{Krishna03} and Ramakrishnan \cite{Ramak02}. 

We will, however, only be referring to the corresponding Langlands $L$-function, which  is  defined by $$L(s, \pi, r^{\pm})=\prod_v L(s,\pi_v, r^{\pm}_v)$$ with $L(s,\pi_v, r^{\pm}_v)=L(s, r^{\pm} \circ \phi_v)$.  For unramified $v$ and $\pi_v$ spherical the latter is given by $$\det(I-r^{\pm}(A(\pi_v)){\rm Nm}(v)^{-s})^{-1},$$ where $A(\pi_v)$ is the semisimple conjugacy class (Satake parameter) in ${}^L(R_{K/K^+} {\rm GL}_m)(\bfC)$ (which is represented by $(\tilde A(\pi_v), \Frob_v)$ for $\tilde A(\pi_v) \in \widehat{R_{K/K^+} {\rm GL}_m}$). For more details and its analytic properties we refer the reader to \cite{GrbacShahidi}.

We will also consider the twisted Asai $L$-function $L(s, \pi, r^{\pm} \otimes \Psi)$ for $\Psi: G_{K^+} \to \ov{\bfQ}_p^*$. For unramified places its Euler factors are given by  $$\det(I-\Psi(\Frob_v) r^{\pm}(A(\pi_v)){\rm Nm}(v)^{-s})^{-1}.$$

The  operation for Galois representations corresponding to the Asai transfer is usually called tensor (or multiplicative) induction: Consider a representation $\rho:G_K \to {\rm GL}(V)$ for an $n$-dimensional vector space $V$ over a field. Then we define the following  extensions of $V \otimes V^c$ to $G_{K^+}$: 

$${\rm As}^{\pm}(\rho):G_{K^+} \to {\rm GL}(V \otimes V^c)$$  given by $${\rm As}^{\pm}(\rho)(g)(x \otimes y)=\rho(g)x \otimes \rho^c(g)(y) \text{ for } g|_K=1$$ and $${\rm As}^{\pm}(\rho)(\tilde c)(x \otimes y)=\pm y \otimes x.$$ 
We first note that ${\rm As}^{-}(\rho)\cong {\rm As}^{+}(\rho) \otimes \chi_{K/K^+}$. We quote further properties from \cite{Prasad92} Lemma 7.1:

\begin{lemma}[Prasad] \label{lemPras}
\begin{enumerate}
\item For representations $\rho_1$ and $\rho_2$ of $G_K$, $${\rm As}^{+}(\rho_1 \otimes \rho_2)={\rm As}^{+}(\rho_1) \otimes {\rm As}^{+}(\rho_2).$$
\item $${\rm As}^{\pm}(\rho^{\vee}) \cong {\rm As}^{\pm}(\rho)^{\vee}$$
\item For a one-dimensional representation $\chi$ of $G_K$, ${\rm As}^{+}(\chi)$ is the one-dimensional representation of $G_{K^+}$ obtained by composing $\chi$ with the transfer map from $G_{K^+}^{\rm ab}$ to $G_K^{\rm ab}$.
\end{enumerate}
\end{lemma}

We also recall from \cite{Krishna12} p.1362-3 how $r^{\pm}$ and ${\rm As}^{\pm}$ are related:
Given $\sigma: G_K \to {\rm GL}_m(\bfC)$ we can define a homomorphism $$\tilde \sigma: G_{K^+} \to {}^L(R_{K/K^+} {\rm GL}_m)(\bfC)=({\rm GL_m}(\bfC) \times {\rm GL}_m(\bfC)) \rtimes G_K$$ as follows:

$$\tilde \sigma(g)= \begin{cases} (\sigma(g), \sigma(\tilde cg\tilde c^{-1}), g) & \text{ if } g \in G_K\\ (\sigma(g\tilde c^{-1}), \sigma(\tilde cg), g) & \text{ if } g \notin G_K.\end{cases}$$

It is now easy to check that $${\rm As}^{\pm}(\sigma)=r^{\pm} \circ \tilde \sigma:G_{K^+} \to {\rm GL}_{m^2}(\bfC).$$

\section{Relating Selmer groups over $K$ and $K^+$}
In this section we record how to relate Bloch-Kato Selmer groups over $K$ and $K^+$: 

Let $L$ either be $K$ or $K^+$.
Consider a  continuous finite-dimensional representation $V$ of $G_{L}$ over $E$. 
 Let $T \subseteq V$ be a
$G_{L}$-stable $\Oo$-lattice and put $W=V/T$ and $W_n=\{x \in W: \varpi^nx=0\}$.
Let $\Sigma$ be a finite set of places of $L$. Following Bloch-Kato (see also \cite{DiamondFlachGuo04} Section 2.1) we define for $M=W, W_n$ the following Selmer
groups:

	\be H^1_{\Sigma}(L,M)={\rm ker}(H^1(L,M) \to \prod_{v \notin \Sigma} H^1(L_v,M)/H^1_f(L_v,M)),\ee where 
$H^1_f(L_v, W):={\rm im}(H^1_f(L_v,V) \to H^1(L_v,W))$ and 
$$H^1_f(L_v,V)=\begin{cases} {\rm ker}(H^1(L_v,V) \to H^1(I_v, V)) & \text{ if } v \nmid p, \infty\\ {\rm ker}(H^1(L_v,V) \to H^1(L_v, V \otimes B_{\rm cris})) & \text{ if } v \mid p\\ 0 & \text{ if } v \mid \infty. \end{cases}$$
When $V$ is short crystalline (i.e. restrictions to $D_v$ for $v \mid p$ described by Fontaine-Laffaille theory as in \cite{DiamondFlachGuo04} Section 1.1.2, but with the more restrictive filtration condition of  \cite{CHT} Section 2.4.1) we refer the reader to \cite{DiamondFlachGuo04} Section 2.1 and \cite{BergerKlosin13} Section 5 for the definition of the local subgroups $H^1_f(L_v, W_n)$ for finite modules using Fontaine-Laffaille theory.

\begin{lemma} \label{Selmerres}
Let $V$ be a $G_{K^+}$-representation with $T$ and $W$ as above. Let $\Sigma^+$ be a finite set of places of $K^+$ and $\Sigma$ the set of places of $K$ above  $\Sigma^+$.  For $p>2$ we have $$H^1_{\Sigma}(K,W) = H^1_{\Sigma}(K,W)^+ \oplus  H^1_{\Sigma}(K,W)^-,$$ where the superscripts indicate the eigenspaces for the action of $c \in {\rm Gal}(K/K^+)$.

Furthermore the restriction map from $G_{K^+}$ to $G_K$ induces isomorphisms $$H^1_{\Sigma^+}(K^+,W) \overset{\sim}{\to} H^1_{\Sigma}(K,W)^+$$ and $$H^1_{\Sigma^+}(K^+,W \otimes \chi_{K/K^+}) \overset{\sim}{\to} H^1_{\Sigma}(K,W)^-.$$

\end{lemma}

\begin{proof}
For this we note that ${\rm ind}_K^{K^+}(V|_{G_K})=V \oplus V \otimes \chi_{K/K^+}$  and recall Shapiro's lemma which for a finite Galois extension of fields $E/F$ tells us that $$H^1(E,M)\cong H^1(F,{\rm ind}_E^F(M))$$ for  any continuous $p$-adic $G_F$-representation $M$ (see \cite{SkinnerUrban14} 3.1.1 and 3.1.2 for a detailed discussion of Shapiro's lemma in the case that $M$ is a discrete module and the compatibility with restrictions at finite places and \cite{SkinnerUrban14} Lemma 3.1 for the analogous statement for the Greenberg Selmer groups in the case of ordinary representations). The statement for  Bloch-Kato Selmer groups follows easily from this and the projection formula $${\rm ind}_{K_w}^{K^+_v} (V \otimes B_{\rm cris}|_{G_{K_w}}) \cong {\rm ind}_{K_w}^{K^+_v} (V) \otimes B_{\rm cris},$$ where $w$ is a place of $K$ lying over a place $v$ of $K^+$, which in turn lies over $p$.
\end{proof}

\section{Polarized Galois representations and signs}

We recall a class of automorphic representations for which Galois representations have been constructed. We follow the exposition in \cite{Thorne15} Section 2 (except for taking the arithmetic Frobenius normalisation when associating the Galois representation):

\begin{definition}
We say that a pair $(\pi, \chi)$ of an automorphic representation $\pi$ of ${\rm GL}_m(\bfA_K)$ and a continuous character $\chi: (K^+)^* \backslash \bfA_{K^+}^* \to \bfC^*$ is RAECSDC (regular, algebraic, essentially conjugate self-dual, cuspidal) if it satisfies the following properties:
\begin{enumerate}
\item $\pi$ is cuspidal.
\item $\pi^{\vee} \cong \pi^c \otimes \chi \circ {\rm Nm}_{K/K^+}$.
\item $\chi_v(-1)=(-1)^m$ for each place $v \mid \infty$ of $K^+$.
\item The infinitesimal character of $\pi_{\infty}$ agrees with the infinitesimal character of an algebraic representation of the group ${\rm Res}_{K/\bfQ} {\rm GL}_m$.
\end{enumerate}
We say that an automorphic representation $\pi$ of $\GL_m(\bfA_K)$ is RACSDC if it satisfies these conditions with $\chi=\chi_{K/K^+}^m$.
\end{definition}
 
\begin{thm} [\cite{Thorne15} Theorem 2.2, culmination of work of many people, see introduction to \cite{ChenevierHarris} and \cite{Cariani2012}, \cite{Cariani2014}] Let $(\pi, \chi)$ be a RAECSDC automorphic representation of ${\rm GL}_m(\bfA_K)$. Then there exists a continuous semisimple representation 
$$\rho_{\pi}:G_K \to \GL_m(\ov{\bfQ}_p)$$ such that
\begin{enumerate}[(a)]
\item $\rho_{\pi}^{\vee} \cong \rho_{\pi}^c  \otimes \epsilon^{1-m} \chi^{\rm gal}|_{G_K}$.
\item $\rho_{\pi}|_{D_v}$ is de Rham for each $v \mid p$.
\item For each finite place $v$ of $K$, we have $${\rm WD}(\rho_{\pi}|_{G_{K_v}})^{\rm F-ss} \cong {\rm rec}^{\vee}_{K_v}(\pi_v \otimes | \det |^{(1-m)/2}),$$ where ${\rm rec}^{\vee}_{K_v}$ is the local Langlands correspondence for $\GL_m(K_v)$ using the arithmetic Frobenius normalisation.
\end{enumerate}
\end{thm}

\begin{rem}
\begin{enumerate}
\item A standard argument using the Baire category
theorem shows that each of the representations is valued in some  finite extension of $\bfQ_p$ (denoted for simplicity by $E$ in the following, suppressing the dependency on $\pi$).
\item
In certain cases, Galois representations have been constructed for irregular $\pi$. See, in particular, the recent papers of  \cite{PilloniStroh} and \cite{GoldringK} in the case that the descent of $\pi$ to a unitary group has a non-degenerate limit of discrete series as archimedean component.
\item For general regular $L$-algebraic cuspidal automorphic representation $\pi$ of $\GL_m(\bfA_K)$ \cite{HLTT} and \cite{Scholze15} have constructed Galois representations $\rho_{\pi} :G_K \to \GL_m(E)$ such that its Frobenius eigenvalues match the Hecke eigenvalues of $\pi$. These are also expected to satisfy local-global compatibility.
\end{enumerate}
\end{rem}

\subsection{Polarizations and signs}
Recall that $K/K^+$ is a CM quadratic extension and $E$ a finite extension of $\bfQ_p$.


\begin{definition}
Let $\Psi: G_{K^+} \to E^*$ be a character 
and $R: G_K \to \GL_m(E)$ an absolutely irreducible representation. We call $(R,\Psi)$ \emph{polarized} if $$R^{\vee} \cong R^c \otimes \Psi|_{G_K}.$$
\end{definition}

In this section we consider complex conjugations $c_v \in G_{K^+}$ for $v \mid \infty$ and conjugate representations 
$R^{c_v}$ defined by $R^{c_v}(g)=R(c_v g c_v)$ (which are all isomorphic to $R^c$ defined using our fixed complex conjugation $\tilde c$).
By Schur's Lemma there exists a matrix $A_v \in \GL_m(\ov{\bfQ}_p)$, unique up to scalar, such that $R^{\vee}=A_vR^{c_v} A_v^{-1} \Psi|_{G_K}$, depending on $v \mid \infty$. This implies that $$R^{c_v}=(A_v^T)^{-1} R^{\vee} A_v^T \Psi^{-1}|_{G_K},$$ using $(\Psi|_{G_K})^{c_v}=\Psi|_{G_K}$. One can then apply Schur's Lemma again (see e.g. \cite{CalegariPoincare}) to deduce that the matrix $A_v$ is in fact symmetric or antisymmetric. 

\begin{definition} \label{signdefn}
Let $(R,\Psi)$ be polarized. If the matrix $A_v$ described above is symmetric we call $(R,\Psi)$ odd at $v$ 
(or $R$ odd at $v$ with respect to $\Psi$) and define ${\rm sign}(R,\Psi, c_v)=1$ (otherwise we put ${\rm sign}(R,\Psi, c_v)=-1$). If $(R,\Psi)$ is odd for all $v \mid \infty$ we call $(R,\Psi)$ \emph{totally odd}.
\end{definition}

\begin{lemma}
The signs  ${\rm sign}(R,\Psi, c_v)$ are independent of $v \mid \infty$ if and only if the values $\Psi(c_v)$ are.
\end{lemma}

\begin{proof} For $R^{\tilde c'}$ the matrix $A_{\tilde c'}$ equals $A_{\tilde c} R(\tilde c' \tilde c)$. Since
\begin{eqnarray*}\left(A_{\tilde c} R(\tilde c' \tilde c) \right)^T&=&R(\tilde c' \tilde c)^T A_{\tilde c}^T=R^{\vee}(\tilde c \tilde c') A_{\tilde c}^T=A_{\tilde c} R(\tilde c' \tilde c)A_{\tilde c}^{-1} A_{\tilde c}^T \Psi(\tilde c \tilde c')\\  &=&A_{\tilde c} R(\tilde c' \tilde c) \cdot {\rm sign}(R,\Psi, \tilde c) \Psi(\tilde c \tilde c') \end{eqnarray*} we deduce that $${\rm sign}(R,\Psi, \tilde c')={\rm sign}(R,\Psi, \tilde c) \Psi(\tilde c \tilde c').$$

\end{proof}

\begin{rem}
This lemma shows that our notion of ``totally odd polarized'' agrees with that in \cite{BLGGT} Section 2, except that we do not demand that $\Psi(c_v)=-1$ for all $v \mid \infty$, which can, however,  easily be achieved by replacing $\Psi$ by $\Psi \chi_{K/K^+}$.
\end{rem}

\begin{example}[Essentially self-dual $G_{K^+}$-representation] \label{ex5.8}
\begin{enumerate}
\item
Let $\tilde R: G_{K^+} \to \GL_m(E)$ absolutely irreducible such that $$\tilde R^{\vee} \cong \tilde R \otimes \Psi.$$ This implies that $\tilde R$  is either (generalized) symplectic or orthogonal,  i.e. that there exists a sympletic or orthogonal pairing for the vector space underlying $\tilde R$. 
\cite{CalegariPoincare} Lemma 2.6 shows that $(\tilde R|_{G_K}, \Psi)$ is odd at $v$ either if $\tilde R$ is symplectic and $\Psi(c_v)=-1$ or if $\tilde R$ is orthogonal and $\Psi(c_v)=1$.

In particular, this shows that the notion of oddness in Definition \ref{signdefn} generalizes the notion of oddness for $2$-dimensional representations $\rho:G_{K^+} \to \GL_2(\ov{\bfQ_p})$ requiring $\det(\rho)(\tilde c)=-1$.
\item Consider a $2$-dimensional irreducible representation $R:G_K \to \GL_2(\ov{\bfQ}_p)$ such that  $R^{\vee} \cong R^c \otimes \det(R)^{-1}$ and ${\rm det} \, R$ extends to $G_{K^+}$. Then $R^c \cong R$, so $R$ extends to $G_{K^+}$. Let $\tilde R: G_{K^+}  \to \GL_2(\ov{\bfQ}_p)$ be either of the two extensions. This satisfies $\tilde R^{\vee} \cong \tilde R \otimes \det(\tilde R)^{-1}$ (but not $\tilde R^{\vee} \cong \tilde R \otimes \det(\tilde R)^{-1} \chi_{K/K^+}$ since $\tilde R \not \cong \tilde R \otimes \chi_{K/K^+}$). This means that ${\rm sign}(R)=1$ is again equivalent to the oddness of the polarization character $\det(\tilde R)$. 
\end{enumerate}
\end{example}

\begin{thm}[Bella{\"{\i}}che-Chenevier \cite{BellaicheCheneviersign},\cite{BLGGT} Theorem 2.1.1(1)] \label{BCsign}
Let $(\pi,\chi)$ be a RAECSDC representation of $\GL_m(\bfA_K)$. Then (any irreducible factor of) $\rho_{\pi}$ is totally odd with respect to $\epsilon^{1-m} \chi^{\rm gal}$.
\end{thm}

\begin{rem} \label{r64}
If $p>2$  and $\rho_{\pi}$ is residually irreducible then the signs of $\rho_{\pi}$ and its reduction agree (see e.g. \cite{BellaicheCheneviersign} section 2.3). This implies that in such cases oddness is also known for irregular $\pi$ if there is a congruence to a RAECSDC representation.

Work of V. Lafforgue should more generally allow to relax the regularity assumption to include $\pi$ such that the descent to the corresponding unitary group has a non-degenerate limit of discrete series as archimedean component.

\end{rem}

\section{Bloch-Kato Conjecture}
Let $\pi$ be a regular $L$-algebraic cuspidal automorphic representation of $\GL_m(\bfA_K)$. As noted above, \cite{HLTT} and \cite{Scholze15} have constructed the Galois representations $\rho_{\pi} :G_K \to \GL_m(E)$ associated to $\pi$ for $E$ a (sufficiently large) finite extension of $\bfQ_p$ (we again take the arithmetic Frobenius normalisation).

In this section we state the Bloch-Kato conjecture (following the exposition in \cite{Dummigan14}) for the Asai $L$-value $L(1,\pi,r^{\pm} \otimes \Psi)$ for $\Psi:G_{K^+} \to E^*$ a continuous Galois character.

Let $F$ be a number field (including the field of definition of the Satake parameters of $\pi$), and $\fp \mid p$ a place of $F$ with $F_{\fp}\cong E$.
Let ${\rm As}^{\pm}(\rho_{\pi}):G_{K^+} \to {\rm GL}_m(E)$ be the Asai plus/minus representation defined in section \ref{Asaidefinition}.
By \cite{Clozel90}  Conjecture 4.5 (applied to the conjectural Asai transfer ${\rm As}^{\pm}(\pi)$ to ${\rm GL}_{m^2}(\bfA)$) there should exists a motive $\mM(\pi, r^{\pm}\otimes \Psi)$ of rank $m^2$ with coefficients in $F$ such that if $\rho_{\mM(\pi, r^{\pm}\otimes \Psi), \fp}:G_{K^+} \to \GL_{m^2}(E)$ is its $\fp$-adic realisation then $\rho_{\mM(\pi, r^{\pm}\otimes \Psi), \fp}(\Frob_v^{-1})$ is conjugate in $\GL_{m^2}(E)$ to $(r^{\pm} \otimes \Psi)(A(\pi_v))$, where $A(\pi_v)$ is the semisimple conjugacy class (Satake parameter) in ${}^L(R_{K/K^+} {\rm GL}_m)(\bfC)$. Note that this $\fp$-adic realisation is dual to ${\rm As}^{\pm}(\rho_{\pi}) \otimes \Psi$ due to the arithmetic Frobenius normalisation we chose for $\rho_{\pi}$.

Put $\mM:=\mM(\pi, r^{\pm}\otimes \Psi)$. 
Let $H_{\rm B}(\mM)$ and $H_{\rm dR}(\mM)$ be the Betti and de Rahm realisations, and let $H_{\rm B}(\mM)^{\pm}$ be the eigenspaces for the complex conjugation $\tilde c$. Following Deligne we call $\mM$ critical if ${\rm dim}(H_{\rm B}(\mM)^+)={\rm dim}(H_{\rm dR}(\mM)/{\rm Fil}^0)$. Assume that this is the case for $\mM$ (which requires, in particular, a choice of $r^+$ or $r^-$, see Proposition \ref{prop}).

Assume $p\gg 0$ (see section 4 of \cite{Dummigan14} for details, but in particular, so that $({\rm As}^{\pm}(\rho_{\pi}) \otimes \Psi)|_{D_v}$ can be described using Fontaine-Laffaille theory) and that $\pi$  is not ramified at any places dividing $p$. Let $\Oo_{(\fp)}$ be the localisation at $\fp$ of the ring of integers $\Oo_{F}$ of $F$. Choose an $\Oo_{(\fp)}$-lattice $T_B^{\pm}$ in $H_{\rm B}(\mM)$ in such a way that $T_{\fp}^{\pm}:=T_B^{\pm} \otimes \Oo$ is a $G_{K^+}$-invariant lattice in the $\fp$-adic realisation. Then choose an  $\Oo_{(\fp)}$-lattice $T_{\rm dR}^{\pm}$ in $H_{\rm dR}(\mM)$ in such a way that $\mathbb{V}(T_{\rm dR}^{\pm} \otimes \Oo)=T_{\fp}^{\pm}$ as $G_{K^+_p}$-representations, where $\mathbb{V}$ is the version of the Fontaine-Laffaille functor used in \cite{DiamondFlachGuo04}.


Let $\Omega$ be a Deligne period scaled according to the above choice, i.e. the determinant of the isomorphism
$$H_{\rm B}(\mM)^+ \otimes \bfC \cong (H_{\rm dR}(\mM)/{\rm Fil}^0) \otimes \bfC,$$ calculated with respect to the bases of $(T_B^{\pm})^+$ and $T_{\rm dR}^{\pm}/{\rm Fil^1}$, so well-defined up to  $\Oo_{(\fp)}^*$.

\begin{conj}[\cite{BlochKato90}, \cite{Dummigan14} Conjecture 4.1] \label{BlochKato}
If $\Sigma^+$ is a finite set of primes, containing all $v$ where $\pi_v$ or $K/K^+$ is ramified, but not containing $p$ 
then
$${\rm ord}_{\fp} \left( \frac{L^{\Sigma^+}(1,\pi,r^{\pm} \otimes \Psi)}{\Omega}\right)={\rm ord}_{\fp} \left( \frac{\# H^1_{\Sigma^+}(K^+,T_{\fp}^* \otimes (E/\Oo))}{\#H^0(K^+,T_{\fp}^*\otimes (E/\Oo))}\right),$$ where $T^*_{\fp}={\rm Hom}_{\Oo}(T_{\fp},\Oo)$, with the dual action of $G_{K^+}$, and $\#$ denotes a Fitting ideal. 
\end{conj}

\begin{rem}
\begin{enumerate}
\item Note that the Langlands $L$-function $L(1,\pi, r^{\pm} \otimes \Psi)$ equals $L(1,\mM)$ 
and relates to the Galois modules on the right hand side given by $\mM^{\vee} \cong {\rm As}^{\pm}(\rho_{\pi}) \otimes \Psi$. We will write the Selmer group in the numerator of the right hand side as  $H^1_{\Sigma^+}(K^+, {\rm As}^{\pm}(\rho_{\pi}) \otimes \Psi \otimes (E/\Oo))$ in the following.

\item If $m=2$  (so that $\rho_{\pi}^{\vee} \cong \rho_{\pi} \otimes \det(\rho_{\pi})^{-1}$) and the central character of $\pi$ extends to $G_{K^+}$ (the situation in \cite{Berger15}) or if $(\pi, \chi)$ is an RAECSDC representation and $\Psi=\epsilon^{1-m}\chi^{\rm gal}$ then ${\rm As}(\rho_{\pi}) \otimes \Psi$ is self-dual (using the properties of the twisted tensor induction recalled in section \ref{Asaidefinition} and  that the transfer of $\Psi|_{G_K}$ to $G_{K^+}$ is $\Psi^2$, see Lemma 4.1 of \cite{Berger15}). 

\end{enumerate}
\end{rem}

For the case of $\Psi|_{G_K}=\epsilon^{-m}$ we analyze the criticality of the Asai motive $\mM(\pi, r^{\pm}\otimes \Psi)=\mM(\pi, r^{\pm})(m)$ (or equivalently its dual ${\rm As}^{\pm}(\rho)(-m)$):

\begin{prop} \label{prop}
Let $p>2$ and $\rho=\rho_{\pi}:G_K \to {\rm GL}_n(E)$ be a residually irreducible continuous representation arising from a regular  motive with coefficients in $F$ that is pure of some weight.  We assume that the Hodge decomposition of $H_B(\mM(\pi,r^{\pm}))$ involves $H^{a,b}$ with $0 \leq a, b \leq 2{\rm wt}(\rho)$. Then for $m \in \bfZ_{>0} \cap [1, {\rm wt}(\rho)]$ (the left-hand side of the critical strip)
the motive  ${\rm As}^{(-1)^{m+1}}(\rho)(-m)$ is critical in the sense of Deligne, whereas ${\rm As}^{(-1)^{m}}(\rho)(-m)$ is not (so exactly the odd (resp. even) integers in $[1, {\rm wt}(\rho)]$ are critical for ${\rm As}^+(\rho)$ (resp. ${\rm As}^-(\rho)$).
\end{prop}

\begin{proof}
We need to show $${\rm dim} H_B(\mathcal{M}(\pi, r^{(-1)^{m+1}})(m))^+={\rm dim} H_{\rm dR}(\mathcal{M}(\pi,r^{(-1)^{m+1}})(m))/{\rm Fil}^0.$$ (In the following we will write $r$ for $r^{(-1)^{m+1}}$.) Since we assume that $m$ lies in the left-hand side of the critical strip and that $\rho$ arises from a regular pure motive we know that  \begin{eqnarray*} {\rm dim} H_{\rm dR}(\mathcal{M}(\pi,r)(m))/{\rm Fil}^0&=&\frac{1}{2}({\rm dim} H_{\rm dR}(\mathcal{M}(\pi,r)(m))-h^{{\rm wt}(\rho),{\rm wt}(\rho)})\\ &=& \frac{n(n-1)}{2}.\end{eqnarray*} 

To calculate ${\rm dim} H_B(\mathcal{M}(\pi, r)(m))^+$ we follow the argument in \cite{Harris13} Section 1.3: Write $M_B:=H_B(\mathcal{M}(\rho))$ for the Betti realisation of the motive corresponding to $\rho$, which is a module over the number field $F$. Choose an $F$-basis $e_1, \ldots e_n$ of $M_B$. Let $t_B=2 \pi it$ be a rational basis for $K^+(1)_B=(2 \pi i) K^+$ and $e_{\chi}$ a basis vector for the Betti realisation of the Artin motive $K^+(\chi_{K/K^+})$ of rank 1 over $K^+$ attached to the character $\chi_{K/K^+}$. As a model of $H_B(\mathcal{M}(\pi, r)(m))$ we take $$H_B(\mathcal{M}(\pi, r)(m))=M_B \otimes M_B^c(m) \otimes K^+(\chi_{K/K^+})_B^{\otimes m+1}$$ with the action \begin{eqnarray*}\tilde c(e_a \otimes e_b^c \otimes t_B^{m}\otimes e_{\chi}^{m+1})&=&e_b \otimes e_a^c \otimes (-1)^m t_B^{m}\otimes (-1)^{m+1} e_{\chi}^{m+1}\\&=&-e_b \otimes e_a^c \otimes t_B^{m}\otimes e_{\chi}^{m+1}.\end{eqnarray*} This implies that the vectors $$\{e_{ab}^+=[e_a \otimes e_b^c -e_b \otimes e_a^c] \otimes t_B^{m}\otimes e_{\chi}^{m+1}, a<b\}$$ form a basis for $H_B(\mathcal{M}(\pi, r^{(-1)^{m+1}})(m))^+$, so that its dimension is therefore also $\frac{n(n-1)}{2}$.
\end{proof}

\section{Construction of elements in critical Selmer groups for general Asai representations} \label{s7}

This section proves the main result Theorem \ref{general} which establishes the Galois part of the proof of one direction of the Bloch-Kato conjecture for general Asai representations (i.e. we assume the existence of a residually reducible representation $R$ that could be constructed by establishing suitable congruences of polarized automorphic forms). Note that the assumption in the theorem on the reduction of $R$ is that $\ov{R}^{\rm ss}$ has two summands that get swapped under the involution $\sigma \mapsto \sigma^{c \vee} \otimes \Psi^{-1}$ corresponding to the polarization of $(R,\Psi)$.

\begin{thm} \label{general}
Let $p>2$, $\Psi:G_{K^+} \to E^*$ a character, and $\rho:G_K \to {\rm GL}_n(E)$ a residually irreducible continuous representation with $$\ov{\rho} \not \equiv \ov{\rho}^{c \vee} \Psi^{-1} \mod{\varpi}.$$ Assume there exists an absolutely irreducible representation $$R: G_K \to {\rm GL}_{2n}(E)$$ that is unramified away from a finite set of places $\Sigma$ and short crystalline at $v \mid p$ such that $$R^{\vee} \cong R^c \otimes {\Psi}|_{G_K}$$ 
and $$\ov{R}^{\rm ss}\cong \ov{\rho} \oplus \ov{\rho}^{c \vee} \otimes \ov{\Psi}^{-1}|_{G_K}.$$ If $\Sigma^+$ is the set of places of $K^+$ lying below $\Sigma$  then $$p \mid \#H^1_{\Sigma^+}(K^+, {\rm As}^{-\Psi(\tilde c){\rm sign}(R, \Psi, \tilde c)}(\rho)\otimes \Psi \otimes (E/\Oo)).$$ 

\end{thm}

\begin{proof}
We generalize the proof of Theorem 6.1 in \cite{Berger15}. 

By assumption $\ov{R}^{\rm ss}\equiv  \ov{\rho} \oplus \ov{\rho}^{\tilde c \vee} \otimes \ov{\Psi}^{-1}|_{G_K} \mod{\varpi}$. By Ribet's lemma (see e.g. Theorem 1.1 of \cite{Urban01}) we know there exists a lattice for $R$ such that \begin{equation} \label{lattice} R\equiv \begin{pmatrix} \ov{\rho}&*\\0&\ov{\rho}^{\tilde c \vee} \ov{\Psi}^{-1}|_{G_K}\end{pmatrix} \mod{\varpi}\end{equation} and this extension is not split.

We claim now that this gives rise to an element in $$H^1_{\Sigma}(K,{\rm As}(\rho)\otimes \Psi \otimes (E/\Oo))^{-\Psi(\tilde c){\rm sign}(R, \Psi, \tilde c)},$$ which is isomorphic to $H^1_{\Sigma^+}(K^+, {\rm As}^{-\Psi(\tilde c){\rm sign}(R, \Psi, \tilde c)}(\rho)\otimes \Psi \otimes (E/\Oo))$ by Lemma \ref{Selmerres}.

First note that
\eqref{lattice} gives a non-trivial class in $H^1(G_K, \Hom_{\bfF}(\ov{\rho}^{\tilde c \vee}\ov{\Psi}^{-1}|_{G_K}, \ov{\rho}))$.  By the assumptions on the ramification and crystallinity of $\tilde R$ we have, in fact, a class in $H^1_{\Sigma}(G_K, \Hom_{\bfF}(\ov{\rho}^{\tilde c \vee}\ov{\Psi}^{-1}, \ov{\rho}))$. Since $$\Hom_{\bfF}(\ov{\rho}^{\tilde c \vee} \ov{\Psi}^{-1}, \ov{\rho})\cong \ov{\rho} \otimes \ov{\rho}^{\tilde c} \otimes\ov{\Psi}\cong ({\rm As}(\ov{\rho})\otimes \ov{\Psi})|_{G_K}$$
 we obtain an element in $$H^1_{\Sigma}(G_K, {\rm As}(\ov{\rho})\otimes \ov{\Psi}) \cong H^1_{\Sigma}(K,{\rm As}(\rho) \otimes \Psi )\otimes (E/\Oo)[\varpi]),$$ which injects into $H^1_{\Sigma}(K,{\rm As}({\rho})\otimes \Psi \otimes (E/\Oo))$ since $H^0(G_{K^+},{\rm As}(\ov{\rho})\otimes \ov{\Psi})$=0. To see the latter, assume that $\Hom_{G_{K^+}}(\mathbf{1}, {\rm As}(\ov{\rho})\otimes \ov{\Psi}) \neq 0$. This implies  $\Hom_{G_{K}}(\mathbf{1}, \ov{\rho} \otimes  \ov{\rho}^c \otimes \ov{\Psi}) \neq 0$, contradicting our assumption that $\ov{\rho}$ and $\ov{\rho}^{c \vee} \ov{\Psi}^{-1}$ are irreducible and non-isomorphic.


\begin{lemma} \label{explicit}
On $H^1(G_K, {\rm As}^+(\ov{\rho})\otimes \ov{\Psi})$ the $\tilde c$-action coincides with $-\Psi(\tilde c)$ times the ``polarization involution" $\perp_{\tilde c}$ arising from the involution on $\bfF[G_K]$ given by $g \mapsto \tau(g):=\tilde cg^{-1}\tilde c^{-1} \ov{\Psi}^{-1}(g)$ for $g \in G_K$.
\end{lemma}

\begin{proof}
We recall from \cite{BellaicheChenevierbook} Section 1.8 for how $\tau$ induces an involution on $H^1(G_K, {\rm As}(\ov{\rho})\otimes \ov{\Psi})$. (For this action we only require the $G_K$-module, so we will write the coefficients as $\Hom_{\bfF}(\ov{\rho}^{\tilde c \vee} \ov{\Psi}^{-1}, \ov{\rho})$.) In our case, $\tau$ is an anti-involution of algebras, and the corresponding involution on representations $\sigma: G_K \to {\rm GL}_m(\bfF)$ 
is given by $\sigma \mapsto \sigma^{\perp}:= (\sigma \circ \tau)^T$. 

\cite{BellaicheChenevierbook} (26) on p.51 explains that the induced involution on $H^1(K,\Hom_{\bfF}(\ov{\rho}^{\tilde c \vee}\ov{\Psi}^{-1}, \ov{\rho})))$ can be described as follows: 
Associate to a cocycle $\phi \in Z^1(G_K, \Hom_{\bfF}(\ov{\rho}^{\tilde c \vee}\ov{\Psi}^{-1}, \ov{\rho})))$ the  representation $$\rho_{\phi}: G_K \to \GL_{2n}(\bfF): g \mapsto \begin{pmatrix}\ov{\rho}(g)&b(g)\\0&\ov{\rho}^{\tilde c \vee}\ov{\Psi}^{-1}(g) \end{pmatrix}$$ with $b(g)=\phi(g) \ov{\rho}^{\tilde c \vee} \ov{\Psi}^{-1}(g)$. Then $\phi^{\perp}:=\phi^{\perp_{\tilde c}}$ given by $$\phi^{\perp}(g):= b^T(\tilde cg^{-1}\tilde c^{-1}) \ov{\rho}^{\tilde c \vee}(g^{-1})$$ defines an involution $\perp$ on  $H^1_{\Sigma}(K, \Hom_{\bfF}(\ov{\rho}^{\tilde c \vee} \ov{\Psi}^{-1}, \ov{\rho})))$.

As in the proof of Lemma 5.2 in \cite{Berger15} we calculate that $$\phi^{\perp}(g)=\ov{\rho}(g) \phi^T(\tilde cg^{-1}\tilde c^{-1}) \rho^{\tilde c \vee}(g^{-1})\ov{\Psi}(g)=-\phi^T(\tilde cg\tilde c^{-1}),$$ where we used the cocycle relation for the last equality.

We compare this to the action of $\tilde c \in G_{K^+}$ on $[\phi] \in H^1_{\Sigma}(K,{\rm As}(\ov{\rho})\otimes \ov{\Psi})$ given by $$(\tilde c.\phi)(g)={\rm As}(\ov{\rho})\otimes \Psi(\tilde c) \phi(\tilde cg \tilde c^{-1}).$$ Since ${\rm As}(\rho)(\tilde c)$ acts as transpose on the upper right shoulder, $$\phi^T \in Z^1(G_K, \Hom_{\bfF}( \ov{\rho}^{\vee}, (\ov{\rho}^{\tilde c \vee} \ov{\Psi}^{-1})^{\vee})=Z^1(G_K,  {\rm As}(\ov{\rho})\otimes \ov{\Psi}|_{G_K}),$$ complex conjugation acts by $\phi \mapsto (g \mapsto \Psi(\tilde c) \phi^T (cgc^{-1}))$ on  a cocyle representing a class $[\phi] \in  H^1_{\Sigma}(K, {\rm As}(\ov{\rho})\otimes \ov{\Psi})$.

\end{proof}

By using a result of Bella{\"{\i}}che and Chenevier we can now show that the extension \eqref{lattice} constructed using  Theorem 1.1 of \cite{Urban01} lies in $H^1(K,{\rm As}(\rho)\otimes \Psi \otimes (E/\Oo))^{-\Psi(\tilde c){\rm sign}(R, \Psi, \tilde c)}$. In fact, Urban' s Theorem 1.1  constructs an element $c_G$ of ${\rm Ext}^1_{\bfF[G_K]}(\ov{\rho}^{\tilde c \vee} \ov{\Psi}^{-1}, \ov{\rho})$ by considering the linear extension $\widehat{{R}}$ of $R$ to $\Oo[G_K] \to M_{2n}(\Oo)$.  Put $T:= \tr(\widehat{{R}})$.  Since $\widehat{{R}}$ is absolutely irreducible \cite{BellaicheChenevierbook} Proposition 1.6.4 tells us that $\ker \widehat{{R}}=\ker T$. The morphism $\widehat{{R}}$ induces a morphism from $\Oo[G_K]/\ker \widehat{{R}}$ and so, applying \cite{Urban01} Theorem 1.1 again, we see that $c_G$ lies in the subspace  $${\rm Ext}^1_{\bfF[G_K]/\ker T}(\ov{\rho}^{\tilde c \vee} \ov{\Psi}^{-1}, \ov{\rho})\subset \Ext^1_{\bfF[G_K]}(\ov{\rho}^{\tilde c \vee} \ov{\Psi}^{-1}, \ov{\rho}).$$
We conclude the proof using \cite{BellaicheChenevierbook} Proposition 1.8.10(i) which tells us that $\perp_{\tilde c}$ acts by multiplication by ${\rm sign}(\tilde R, \Psi, \tilde c)$ 
on $\Ext^1_{\bfF[G_K]/\ker T}(\ov{\rho}^{\tilde c \vee} \ov{\Psi}^{-1}, \ov{\rho})$.

\end{proof}

Under some mild additional assumptions one can check that the Theorem implies that this construction always produces elements in the Selmer group for the Asai motive that is critical in the sense of Deligne. 

For $\Psi|_{G_K}=\epsilon^{-w}$ (e.g. the polarization character for the Galois representations associated to RACSDC representations of ${\rm GL}_{w+1}(\bfA_K)$) and ${\rm sign}(R)=1$ we have ${\rm As}^{-\Psi(\tilde c){\rm sign}(R)}(\rho(i))\otimes \Psi={\rm As}^{(-1)^{w+1}}(\rho)(-w+2i)$. We can therefore deduce the following corollary to Theorem \ref{general} and Proposition \ref{prop}:

\begin{cor} \label{cor}
Let $p>2$, $K/K^+$ a CM field and $\rho:G_K \to {\rm GL}_n(E)$ be an residually irreducible continuous representation arising from a regular motive that is pure of some weight.  Consider an absolutely irreducible representation $$R: G_K \to {\rm GL}_{2n}(E)$$ with ${\rm sign}(R)=1$ that is pure of weight $w$, unramified away from a finite set of places $\Sigma$ and short crystalline at $v \mid p$ and satisfies $$R^{\vee} \cong R^c \otimes \Psi|_{G_K}$$ with $\Psi:G_K \to E^*$ such that $\Psi|_{G_K}=\epsilon^{-w}$. Assume that for some $i \in \bfZ$  such that $w-2i$ lies in the left-hand side of the critical strip of ${\rm As}(\rho)$ we have
$$\ov{R}^{\rm ss}=\ov{\rho}(i) \oplus \ov{\rho}^c(i)^{\vee} \otimes \ov{\Psi}^{-1}|_{G_K}$$ and  $$\ov{\rho}(i) \not \equiv \ov{\rho}^c(i)^{\vee} \ov{\Psi}^{-1} \mod{\varpi}.$$ 
If $\Sigma^+$ is the set of places of $K^+$ lying below $\Sigma$  then $$p \mid \#H^1_{\Sigma^+}(K^+, {\rm As}^{(-1)^{w+1}}(\rho)(-w+2i) \otimes (E/\Oo))$$  and ${\rm As}^{(-1)^{w+1}}(\rho)(-w+2i)$ is critical.
\end{cor}

\begin{rem}
By the Bloch-Kato conjecture \ref{BlochKato} the Selmer group in the corollary should be governed by the $L$-value (on the right hand side of the critical strip) $L(1+w-2i, \pi,r^{(-1)^{w+1}}).$  For polarized representations $\rho$ (i.e. when $\rho^c=\rho^{\vee}({\rm wt}(\rho))$) the Deligne conjecture for the corresponding $L$-values is discussed in \cite{Harris13} Section 1.3 (where the Asai representation is viewed as an extension of the adjoint representation $\rho \otimes \rho^{\vee}$). (Note that because of our assumption that $\ov{\rho} \not \equiv \ov{\rho}^{c \vee} \otimes \ov{\Psi}^{-1} \mod{\varpi}$ these could be studied in Corollary \ref{cor} only for $i \neq 0$.)
\end{rem}

\section{Fontaine-Mazur conjecture}

By combining Fontaine-Mazur's conjecture (\cite{FontaineMazur} Conjecture  1) with Langlands automorphy conjectures one arrives at the following conjecture:
\begin{conj} \label{conj}
Let $R:G_K \to \GL_m(\ov{\bfQ_p})$ be a continuous irreducible potentially semi-stable representation. Let $\Psi:G_{K^+} \to \ov{\bfQ}_p^*$ be a continuous character such that 
$$R^{\vee} \cong R^c \otimes \Psi|_{G_K}.$$
Then there exists an algebraic essentially conjugate self-dual cuspidal automorphic representation $\pi$ of $\GL_m/K$ such that $\rho_{\pi} \cong R$.
\end{conj}

\begin{rem}
Results towards this conjecture have been proven by \cite{BLGGT} and Thorne \cite{Thorne15}:

\cite{BLGGT} Theorem 4.5.1 proves potential automorphy for odd $R$ that are residually irreducible (and satisfying some further conditions, e.g. potential diagonalisability at $v \mid p$).

\cite{Thorne15} Theorem 7.1 proves the conjecture for certain residually reducible $R$ with $$R^{\vee} \cong R^c \epsilon^{1-m}$$ and $\ov{R}^{\rm ss} \cong \ov{\rho}_1 \oplus \ov{\rho}_2$ with $\ov{\rho}_i|_{G_K(\zeta_p)}$ adequate and $\ov{R}$ of \emph{Schur type}, i.e. $\ov{\rho}_1 \not \cong \ov{\rho}_2$ and $\epsilon^{1-m}\ov{\rho}_1^{c} \not \cong \ov{\rho}_2^{\vee}$ (so residual constituents do \emph{not} get swapped by involution, different to our set-up). 
Apart from some other technical conditions \cite{Thorne15} assumes $R$ ordinary at $v \mid p$ and 
that $\ov{R}^{\rm ss}$ is automorphic, but not that $R$ is odd. The deformation theory techniques used to study residual representations of Schur type do not apply in our setting.
\end{rem}

Similar to Calegari's result (see Theorem \ref{cal}) the oddness of automorphic Galois representations (Theorem \ref{BCsign} and Remark \ref{r64}) and Conjecture \ref{conj} implies the non-existence of even geometric representations (except for certain exceptional cases like Tate twists of even representations with finite image). In the residually reducible case our result allows the following evidence towards conjecture \ref{conj}:

\begin{cor} \label{cor8.3}
Let $(R, \Psi)$ be polarized as in Theorem \ref{general}, in particular such that $$\ov{R}^{\rm ss}=\ov{\rho} \oplus \ov{\rho}^{c \vee} \otimes \Psi^{-1}|_{G_K}.$$ If $$p \nmid \#H^1_{\Sigma^+}(K^+, {\rm As}^{\Psi(\tilde c)}(\ov{\rho})\otimes \ov{\Psi})$$ then $(R,\Psi)$ is totally odd.
\end{cor}

\begin{rem}
\begin{enumerate}
\item Note that we can deduce total oddness by applying Theorem \ref{general} for all complex conjugations $c_v$ for $v \mid \infty$, without a priori assuming the independence of $\Psi(c_v)$, since $H^1_{\Sigma^+}(K^+, {\rm As}^{\Psi(c_v)}(\rho)\otimes \Psi \otimes (E/\Oo))$ is independent of $v \mid \infty$. For this we note that the Galois representations ${\rm As}^{\pm}(\rho)$ (defined as extensions of $\rho \otimes \rho^{c_v}$) are isomorphic for different $c_v$, and that ${\rm As}^{\Psi(c_v)}(\rho)\otimes \Psi$ is independent of $\Psi(c_v)$. 
\item   Corollary \ref{cor8.3} lets us give an alternative proof of Proposition \ref{baby}: Let $\sigma:G_{\bfQ} \to \GL_2(\ov{\bfQ}_p)$ be a short crystalline representation as in Proposition \ref{baby} and choose $K/\bfQ$ imaginary quadratic such that $R:=\sigma|_{G_K}$ is irreducible (i.e. such that $\sigma \not \cong \sigma \otimes \chi_{K/\bfQ}$).  We now take $K^+=\bfQ$, $n=1$, $\rho$ the trivial character $\mathbf{1}_{G_K}:G_K \to \bfZ_p^*$, and $\Psi=\epsilon^{-m}$ in Corollary \ref{cor8.3}.

Note that ${\rm As}^+(\mathbf{1}_{G_K})=\mathbf{1}_{G_{\bfQ}}$ and ${\rm As}^{\Psi(\tilde c)}(\rho)\otimes \Psi= \chi_{K/\bfQ}^m \epsilon^{-m}$. Assume that $m$ is even (if $m$ is odd then there is nothing to show). Vandiver's conjecture then tells us that $p \nmid {\rm Cl}(\bfQ(\mu_p))(\ov{\epsilon}^{-m})$. Since $\bfQ(\mu_p)$ is a ramified extension of $\bfQ(\mu_p^{\otimes m})$ we deduce that $p \nmid {\rm Cl}(\bfQ(\mu_p^{\otimes m}))(\ov{\epsilon}^{-m})$ also. The latter is isomorphic to $H^1_f(\bfQ,\bfQ_p/\bfZ_p(\ov{\epsilon}^{-m}))$ (see e.g. \cite{Rubin00} Proposition 1.6.2). 
For $\Sigma^+=\{p\}$ we therefore have $$ p \nmid \#H^1_f(\bfQ,\bfQ_p/\bfZ_p(\ov{\epsilon}^{-m}))[p]=\#H^1_{\Sigma^+}(\bfQ, {\rm As}^{\Psi(\tilde c)}(\ov{\rho})\otimes \ov{\Psi}),$$ so Corollary \ref{cor8.3} implies that $R$ is odd with respect to $\Psi=\epsilon^{-m}$. As explained in Example \ref{ex5.8} this would require $\Psi(\tilde c)=-1$, a contradiction.
\item
If $(R,\Psi)$ satisfy further the assumptions of Corollary \ref{cor} (but without assuming that $(R, \Psi)$ is odd) then the Asai representation in Corollary \ref{cor8.3} is not critical. 
As explained in (2), in the particular case of $\epsilon^{-m}$ for even $m$ Vandiver's conjecture predicts that  the non-divisibility assumption in Corollary \ref{cor8.3} is satisfied.

\end{enumerate}
\end{rem}

\bibliographystyle{amsalpha}
\bibliography{paramod}

\providecommand{\bysame}{\leavevmode\hbox to3em{\hrulefill}\thinspace}
\providecommand{\MR}{\relax\ifhmode\unskip\space\fi MR }
\providecommand{\MRhref}[2]{%
  \href{http://www.ams.org/mathscinet-getitem?mr=#1}{#2}
}
\providecommand{\href}[2]{#2}
\begin{thebibliography}{BLGGT14}

\bibitem[BC09]{BellaicheChenevierbook}
J.~Bella{\"{\i}}che and G.~Chenevier, \emph{$p$-adic families of {G}alois
  representations and higher rank {S}elmer groups}, Ast\'erisque (2009),
  no.~324.

\bibitem[BC11]{BellaicheCheneviersign}
Jo{\"e}l Bella{\"{\i}}che and Ga{\"e}tan Chenevier, \emph{The sign of {G}alois
  representations attached to automorphic forms for unitary groups}, Compos.
  Math. \textbf{147} (2011), no.~5, 1337--1352.

\bibitem[Ber15]{Berger15}
Tobias Berger, \emph{On the {B}loch-{K}ato conjecture for the {A}sai
  $l$-function}, Preprint available at \url{http://arxiv.org/abs/1507.00684}.

\bibitem[BK90]{BlochKato90}
S.~Bloch and K.~Kato, \emph{{$L$}-functions and {T}amagawa numbers of motives},
  The Grothendieck Festschrift, Vol.\ I, Progr. Math., vol.~86, Birkh\"auser
  Boston, Boston, MA, 1990, pp.~333--400.

\bibitem[BK13]{BergerKlosin13}
T.~Berger and K.~Klosin, \emph{On deformation rings of residually reducible
  {G}alois representations and ${R}={T}$ theorems}, Math. Ann. \textbf{355}
  (2013), no.~2, 481--518.

\bibitem[BLGGT14]{BLGGT}
Thomas Barnet-Lamb, Toby Gee, David Geraghty, and Richard Taylor,
  \emph{Potential automorphy and change of weight}, Ann. of Math. (2)
  \textbf{179} (2014), no.~2, 501--609.

\bibitem[Cal11]{CalegariPoincare}
Frank Calegari, \emph{Even {G}alois representations ({L}ecture notes {P}oincare
  {I}nstitute)}, Preprint available at
  \url{http://www.math.northwestern.edu/~fcale/papers/FontaineTalk-Adjusted.pdf}.

\bibitem[Cal12]{Calegari2012}
Frank Calegari, \emph{Even {G}alois representations and the {F}ontaine--{M}azur
  conjecture. {II}}, J. Amer. Math. Soc. \textbf{25} (2012), no.~2, 533--554.

\bibitem[Car12]{Cariani2012}
Ana Caraiani, \emph{Local-global compatibility and the action of monodromy on
  nearby cycles}, Duke Math. J. \textbf{161} (2012), no.~12, 2311--2413.

\bibitem[Car14]{Cariani2014}
\bysame, \emph{Monodromy and local-global compatibility for {$l=p$}}, Algebra
  Number Theory \textbf{8} (2014), no.~7, 1597--1646.

\bibitem[CH13]{ChenevierHarris}
Ga{\"e}tan Chenevier and Michael Harris, \emph{Construction of automorphic
  {G}alois representations, {II}}, Camb. J. Math. \textbf{1} (2013), no.~1,
  53--73.

\bibitem[CHT08]{CHT}
Laurent Clozel, Michael Harris, and Richard Taylor, \emph{Automorphy for some
  {$l$}-adic lifts of automorphic mod {$l$} {G}alois representations}, Publ.
  Math. Inst. Hautes \'Etudes Sci. (2008), no.~108, 1--181, With Appendix A,
  summarizing unpublished work of Russ Mann, and Appendix B by Marie-France
  Vign{\'e}ras.

\bibitem[Clo90]{Clozel90}
Laurent Clozel, \emph{Motifs et formes automorphes: applications du principe de
  fonctorialit\'e}, Automorphic forms, {S}himura varieties, and
  {$L$}-functions, {V}ol.\ {I} ({A}nn {A}rbor, {MI}, 1988), Perspect. Math.,
  vol.~10, Academic Press, Boston, MA, 1990, pp.~77--159.

\bibitem[DFG04]{DiamondFlachGuo04}
F.~Diamond, M.~Flach, and L.~Guo, \emph{The {T}amagawa number conjecture of
  adjoint motives of modular forms}, Ann. Sci. \'Ecole Norm. Sup. (4)
  \textbf{37} (2004).

\bibitem[Dum15]{Dummigan14}
Neil Dummigan, \emph{Eisenstein congruences for unitary groups}, Preprint
  available at \url{http://neil-dummigan.staff.shef.ac.uk/unitarycong4.pdf}.

\bibitem[FM95]{FontaineMazur}
Jean-Marc Fontaine and Barry Mazur, \emph{Geometric {G}alois representations},
  Elliptic curves, modular forms, \& {F}ermat's last theorem ({H}ong {K}ong,
  1993), Ser. Number Theory, I, Int. Press, Cambridge, MA, 1995, pp.~41--78.

\bibitem[GK15]{GoldringK}
Wushi Goldring and Jean-Stefan Koskivirta, \emph{Strata {H}asse invariants,
  {H}ecke algebras and {G}alois representations}, Preprint available at
  \url{http://arxiv.org/abs/1507.05032}.

\bibitem[GS15]{GrbacShahidi}
Neven Grbac and Freydoon Shahidi, \emph{Endoscopic transfer for unitary groups
  and holomorphy of {A}sai {$L$}-functions}, Pacific J. Math. \textbf{276}
  (2015), no.~1, 185--211.

\bibitem[Har13]{Harris13}
Michael Harris, \emph{{$L$}-functions and periods of adjoint motives}, Algebra
  Number Theory \textbf{7} (2013), no.~1, 117--155.

\bibitem[HLTT16]{HLTT}
Michael Harris, Kai-Wen Lan, Richard Taylor, and Jack Thorne, \emph{On the
  rigid cohomology of certain {S}himura varieties}, Res. Math. Sci. \textbf{3}
  (2016), 3:37.

\bibitem[Kha00]{Khare00}
Chandrashekhar Khare, \emph{Notes on {R}ibet's converse to {H}erbrand},
  Cyclotomic fields and related topics ({P}une, 1999), Bhaskaracharya
  Pratishthana, Pune, 2000, pp.~273--284.

\bibitem[Kri03]{Krishna03}
M.~Krishnamurthy, \emph{The {A}sai transfer to {$\rm GL_4$} via the
  {L}anglands-{S}hahidi method}, Int. Math. Res. Not. (2003), no.~41,
  2221--2254.

\bibitem[Kri12]{Krishna12}
\bysame, \emph{Determination of cusp forms on {$GL(2)$} by coefficients
  restricted to quadratic subfields (with an appendix by {D}ipendra {P}rasad
  and {D}inakar {R}amakrishnan)}, J. Number Theory \textbf{132} (2012), no.~6,
  1359--1384.

\bibitem[Pra92]{Prasad92}
Dipendra Prasad, \emph{Invariant forms for representations of {${\rm GL}_2$}
  over a local field}, Amer. J. Math. \textbf{114} (1992), no.~6, 1317--1363.

\bibitem[PS16]{PilloniStroh}
Vincent Pilloni and Beno{\^{\i}}t Stroh, \emph{Cohomologie coh\'erente et
  repr\'esentations {G}aloisiennes}, Ann. Math. Qu\'e. \textbf{40} (2016),
  no.~1, 167--202.

\bibitem[Ram02]{Ramak02}
Dinakar Ramakrishnan, \emph{Modularity of solvable {A}rtin representations of
  {${\rm GO}(4)$}-type}, Int. Math. Res. Not. (2002), no.~1, 1--54.

\bibitem[Ram04]{Ramak04}
\bysame, \emph{Algebraic cycles on {H}ilbert modular fourfolds and poles of
  {$L$}-functions}, Algebraic groups and arithmetic, Tata Inst. Fund. Res.,
  Mumbai, 2004, pp.~221--274.

\bibitem[Rib76]{Ribet76}
Kenneth~A. Ribet, \emph{A modular construction of unramified {$p$}-extensions
  of {$\bfQ(\mu _{p})$}}, Invent. Math. \textbf{34} (1976), no.~3, 151--162.

\bibitem[Rub00]{Rubin00}
Karl Rubin, \emph{Euler systems}, Annals of Mathematics Studies, vol. 147,
  Princeton University Press, Princeton, NJ, 2000, Hermann Weyl Lectures. The
  Institute for Advanced Study.

\bibitem[Sch15]{Scholze15}
Peter Scholze, \emph{On torsion in the cohomology of locally symmetric
  varieties}, Ann. of Math. (2) \textbf{182} (2015), no.~3, 945--1066.

\bibitem[SU14]{SkinnerUrban14}
Christopher Skinner and Eric Urban, \emph{The {I}wasawa main conjectures for
  {$\rm GL_2$}}, Invent. Math. \textbf{195} (2014), no.~1, 1--277.

\bibitem[Tho15]{Thorne15}
Jack~A. Thorne, \emph{Automorphy lifting for residually reducible {$l$}-adic
  {G}alois representations}, J. Amer. Math. Soc. \textbf{28} (2015), no.~3,
  785--870.

\bibitem[Urb01]{Urban01}
E.~Urban, \emph{Selmer groups and the {E}isenstein-{K}lingen ideal}, Duke Math.
  J. \textbf{106} (2001), no.~3, 485--525.

\end{thebibliography}

\end{document}